\documentclass[12pt,final]{article}
\usepackage{amsmath,amsthm,amsfonts,amssymb,graphicx,enumerate,psfrag}
\usepackage{mathrsfs}
\usepackage{fullpage}

\usepackage[colorlinks,citecolor=blue,urlcolor=blue]{hyperref}
\usepackage[utf8]{inputenc} 
\newtheorem{lemma}{Lemma}
\newtheorem{assume}{Assumption}
\newtheorem{theorem}{Theorem}

\newtheorem{corollary}{Corollary}

\newcommand{\dN}{\mathbb {N}}
\newcommand{\dZ}{\mathbb {Z}}
\newcommand{\dR}{\mathbb {R}}

\newcommand{\dist}{\mathrm{d}}
\newcommand{\cS}{{\mathbb{S}}}
\newcommand{\EE}{{\mathbb{E}}}
\newcommand{\MM}{{\mathbb{M}}}

\newcommand{\mm}{{\mathrm{m}}}

\newcommand{\PP}{{\mathbb{P}}}

\newcommand{\fS}{\mathfrak{S}}
\newcommand{\cP}{\mathcal{P}}

\newcommand{\cH}{\mathcal{H}}
\newcommand{\cc}{\kappa}
\newcommand{\cX}{\mathcal {X}}
\newcommand{\dtv}{{\rm d}_{\textsc{tv}}}

\newcommand{\lip}{\mathrm{Lip}}

\title{Entropy and curvature: beyond the Peres-Tetali conjecture}
\author{Pietro Caputo, Florentin M\"unch and Justin Salez}
\begin{document}
\maketitle
\begin{abstract}
We study Markov chains with non-negative sectional curvature on finite metric spaces. Neither reversibility, nor the restriction to a particular combinatorial distance are imposed. In this level of generality, we prove that a 1-step contraction in the Wasserstein distance implies a 1-step contraction in relative entropy, by the same amount. Our result substantially strengthens a recent breakthrough of the second author, and  has the advantage of being applicable to arbitrary scales. This leads to a time-varying refinement of the standard Modified Log-Sobolev Inequality (MLSI), which allows us to leverage the well-acknowledged fact that \emph{curvature improves at large scales}. We illustrate this principle with several applications, including birth and death chains, colored exclusion processes,  permutation walks, Gibbs samplers for high-temperature spin systems, and attractive zero-range dynamics. In particular, we prove a MLSI with constant equal to the minimal rate increment for the mean-field zero-range process, thereby answering  a long-standing question. 
\end{abstract}
\section{Introduction}
\label{sec:intro}
\paragraph{Geometric contraction.} Throughout the paper, we fix a finite metric space $(\cX,\dist)$ and an irreducible stochastic matrix $P$ on $\cX$. We write $\cP(\cX)$ for the set of probability measures on $\cX$, and equip it with the Wasserstein distance $W(\mu,\nu)  :=  \min_{X\sim\mu,Y\sim \nu}\EE[\dist(X,Y)]$. Following Ollivier \cite{ollivier2009ricci,ollivier2010survey}, we define the \emph{curvature} of $P$ as  the largest number $\cc=\cc(P)\in\dR$ such that 
\begin{eqnarray}
\label{GI}
\forall \mu,\nu\in\cP(\cX),\qquad W(\mu P,\nu P) & \le & (1-\cc)W(\mu,\nu).
\end{eqnarray}
 By the convexity of $W(\cdot,\cdot)$, it is in fact enough to verify this inequality when $\mu$ and $\nu$ are Dirac masses, a task which is simple enough to yield sharp lower bounds on the curvature of many concrete Markov chains. When positive, such lower bounds have been shown to provide systematic control on a number of essential quantitative features of the chain, including geometry \cite{jost2019Liouville,munch2019non, munch2023ollivier}, mixing times \cite{646111,MR2316551}, expansion \cite{salez2021sparse, Muench2023,salez2023spectral}, concentration of measure \cite{MR2683634,joulin2007poisson,eldan2017transport}, spectral independence \cite{MR4415182} and even the cutoff phenomenon \cite{salez2021cutoff}.
 
\paragraph{Entropic contraction.} The purpose of the present paper is to investigate the relation between the \emph{geometric} contraction (\ref{GI}) and its natural \emph{entropic} counterpart:
\begin{eqnarray}
\label{EI}
\forall \mu\in\cP(\cX),\qquad H(\mu P\,|\,\pi) & \le & (1-\cc)H(\mu \,|\,\pi),
\end{eqnarray}
where $\pi$ is the unique invariant law of $P$, and where
$
H(\mu\,|\,\pi) :=  \sum_{x}\mu(x)\log(\mu(x)/\pi(x))
$
denotes the relative entropy (or Kullback-Leibler divergence) of $\mu$ with respect $\pi$. As explained in the lecture notes \cite{entropy}, this discrete-time entropic contraction is stronger than the classical \emph{Modified Log-Sobolev Inequality} (MLSI) of \cite{MR2120475}, which applies to the continuous-time Markov semi-group $(P_t)_{t\ge 0}$ generated by $P-{\rm I}$ and asserts that
\begin{eqnarray}
\label{MLSI}
\forall \mu\in\cP(\cX),\qquad\forall t\ge 0,\qquad H(\mu P_t\,|\,\pi) & \le & e^{-\cc t}\, H(\mu\,|\,\pi).
\end{eqnarray}
Let us state our main result straightaway, and provide additional motivation afterwards.

\paragraph{Main result.} We let $P^\star$ denote the adjoint of $P$ in the Hilbert space $L^2(\cX,\pi)$.  Following \cite{ollivier2010survey,MR4149342,pedrotti2023contractive,munch2023ollivier} (see also \cite{MR4688313} for a related notion), we make  the following structural assumption, which will be shown to hold in many important examples  in Section \ref{sec:applications}. 
\begin{assume}[Non-negative sectional curvature]\label{assume:sectional}For each pair of states $(x,y)\in \cX^2$, there is a coupling $(X^\star,Y^\star)$ of  $P^\star(x,\cdot)$ and $P^\star(y,\cdot)$ such that almost-surely,
\begin{eqnarray*}
\dist(X^\star,Y^\star) & \le & \dist(x,y).
\end{eqnarray*}
\end{assume}
\begin{theorem}[Main result]\label{th:main}Under Assumption \ref{assume:sectional}, the geometric contraction (\ref{GI}) implies the entropic contraction (\ref{EI}) with the same constant $\cc$. \end{theorem}

\paragraph{Relation to previous works.}   The far-reaching possibility that the Ollivier-Ricci curvature might, under appropriate assumptions, be powerful enough to control the rate of exponential decay of the relative entropy 
emerged in the community at the beginning of the present millennium. It  became informally known as the \emph{Peres-Tetali conjecture}, and was investigated by several authors \cite{eldan2017transport,MR4415182,liu2021coupling,munch2023ollivier}, see also \cite{hermon2018characterization,conforti2022probabilistic} for related work on establishing relative entropy decay via probabilistic techniques. In particular, a breakthrough was very recently made by the second author, who managed to deduce the MLSI (\ref{MLSI}) from the geometric contraction (\ref{GI}) and Assumption \ref{assume:sectional},  under the additional restrictions that $P$ is reversible and that the underlying metric is the  combinatorial distance $\dist(x,y)=\min\{n\in\dN\colon P^n(x,y)>0\}$, see  \cite[Theorem 4.4]{munch2023ollivier}. Our Theorem \ref{th:main} strengthens this result in three important ways:
\begin{enumerate}
\item We crucially improve the MLSI (\ref{MLSI}) to the $1-$step entropic contraction (\ref{EI}). 
\item Our metric $\dist$ is arbitrary, thereby considerably broadening the scope of Assumption \ref{assume:sectional}.
\item We do not require the reversibility condition $P^\star=P$. 
\end{enumerate}
The practical interest of each of those improvements will be demonstrated in Section \ref{sec:applications}, where Theorem \ref{th:main} is applied to several important Markov chains. The main motivation for our work was the observation that many natural examples satisfying Assumption \ref{assume:sectional} are actually \emph{flat} in the sense that $\cc(P)=0$, making the MLSI (\ref{MLSI}) useless. In contrast, our stronger conclusion (\ref{EI}) has the advantage of being applicable to arbitrary scales,  thereby allowing us to leverage the long-acknowledged fact that \emph{curvature improves at large scales}.  More precisely, applying Theorem \ref{th:main} to $P_t=e^{t(P-\rm{I})}$ (which also satisfies Assumption \ref{assume:sectional}) instead of $P$ readily yields
\begin{eqnarray}
\label{timevarying}
\forall \mu\in\cP(\cX),\qquad\forall t\ge 0,\qquad H(\mu P_t\,|\,\pi) & \le &\left(1-\cc(P_t)\right)\, H(\mu\,|\,\pi).
\end{eqnarray}
This is always at least as good as the uniform estimate (\ref{MLSI}), since 
$1-\cc(P_t)  \le  e^{-t\cc(P)}$. In fact, the sub-multiplicativity of the function $t\mapsto 1-\cc(P_t)$ ensures that our estimate can only improve as $t$ increases. The benefit can be considerable, as we will see in Section \ref{sec:applications}. 

\paragraph{Other curvature notions.} In recent years, there has been increasing interest in discrete Ricci curvature notions. Indeed, many approaches can be traced back to the 80s.
The Wasserstein contraction which Ollivier famously interpreted as Ricci curvature, was already used by Dobrushin and Shlosman in \cite{dobrushin1985constructive}, and is also known as Dobrushin-Shlosman criterion.
A fundamentally different approach to discrete curvature was taken by Forman who used a discrete Bochner-Weitzenböck decomposition to establish a Ricci curvature notion on cell complexes \cite{forman2003bochner}. Surprisingly, the Ollivier curvature coincides with the Forman curvature when choosing the two-cells optimally \cite{jost2021characterizations,tee2021enhanced}.

Many results in Riemannian geometry critically depend on the dimension. In the discrete counterpart, no meaningful dimension parameter could be found for the Ollivier curvature. 
However, the Bochner formula together with the Bakry-\'Emery calculus \cite{bakry2006diffusions} provides a framework in which a dimension parameter can be naturally introduced in the discrete setting. This lead to the discrete Bakry-\'Emery curvature, introduced independently thrice in \cite{elworthy1991manifolds,schmuckenschlager1998curvature,
lin2010ricci}.
In order to prove Li-Yau and log-Sobolev inequalities, various non-linear modifications of Bakry-\'Emery curvature were introduced in \cite{bauer2015li,munch2014li,Dier2021,weber2021entropy,
erbar2012ricci,MR4688313}, motivated by a lack of a discrete Laplacian chain rule.
Bakry-\'Emery curvature was recently lifted to cell complexes in \cite{munch2024intertwining}.

A new version of entropic curvature was introduced by Rapaport and Samson in 
\cite{rapaport2023criteria}, including local criteria. This however, does not coincide with the entropic curvature by Erbar and Maas \cite{erbar2012ricci}.
Entropic curvature has been proven to be a powerful tool to investigate mixing of interacting particle systems \cite{erbar2018poincare,MR3513606,
erbar2017ricci,erbar2019entropic}, despite the fact that entropic curvature, as a non-linear optimization problem, is hard to compute explicitly.

One key difference between Ollivier curvature and Bakry-\'Emery and entropic curvature lies in the implicit gradient and distance notions. For Ollivier curvature, there is the freedom to choose an arbitrary distance. For Bakry-\'Emery and entropic curvature in contrast, the gradient is implicitly determined by the Markov chain.

Further non-local curvature notions which are easy to compute have been introduced in \cite{devriendt2022discrete, steinerberger2023curvature}.
However, not much theory has been developed yet.

\paragraph{Local reduction.} In order to conclude that a given matrix $P$ satisfies the entropic contraction (\ref{EI}), Theorem \ref{th:main}  requires us to provide, for each pair of states $(x,y)\in\cX^2$:
\begin{enumerate} 
\item[(i)]a coupling $(X,Y)$ of $P(x,\cdot)$ and  $P(y,\cdot)$ such that
$
\EE\left[\dist(X,Y)\right] \le(1-\cc)\dist(x,y);
$
\item[(ii)] a coupling $(X^\star,Y^\star)$ of $P^\star(x,\cdot)$ and  $P^\star(y,\cdot)$ such that
$
\PP\left(\dist(X^\star,Y^\star) \le \dist(x,y)\right) =  1.
$
\end{enumerate}
By virtue of the  so-called \emph{Gluing Lemma} (see, e.g., \cite[Lemma 7.6]{MR1964483}), we may in fact restrict this double task to pairs $(x,y)$ in a subset ${\mathcal S}\subseteq\cX^2$ which  \emph{generates} the metric $\dist$ in the following sense: for each $(x,y)\in\cX^2$, we can write
\begin{eqnarray}
\label{geodesic}
\dist(x,y) & = & \sum_{i=1}^n\dist(x_{i-1},x_i),
\end{eqnarray}
for some $n\in\dN$ and some sequence $(x_0,\ldots,x_n)\in\cX^{n+1}$ such that $x_0=x$, $x_n=y$ and $(x_{i-1},x_i)\in\mathcal S$ for $1\le i\le n$.  Of course, the trivial choice $\mathcal S=\cX^2$ always meets this requirement, but we will see in Section \ref{sec:applications} that many natural metrics are actually generated by much smaller sets, making this reduction quite useful in practice. 
\paragraph{Optimizing the metric.} Let us close this short introduction with an interesting question, inspired by the recent work \cite{salez2023spectral}. Observe that the geometric contraction (\ref{GI}) depends on the underlying metric $\dist$, whereas the entropic contraction (\ref{EI}) does not. This asymmetry can  be turned to one's advantage by treating the metric $\dist$ as a variable which one can try to fine-tune so as to optimize the resulting constant. More precisely, our result shows that \emph{any} irreducible stochastic matrix $P$ satisfies the entropic contraction  (\ref{EI}) with constant
\begin{eqnarray*}
\cc_\star(P) & := & \sup_{\dist\in{\MM(P)}} \cc(P,\dist),
\end{eqnarray*}
where ${\MM}(P)$ is the set of all metrics on $\cX$ under which Assumption \ref{assume:sectional} holds, and  where $\cc(P,\dist)$ denotes the curvature of $P$ with respect to the metric $\dist$. 
In view of the recent work \cite{salez2023spectral} on monotone chains, it is natural to ask for an effective characterization of  $\cc_\star$, at least under appropriate structural assumptions on $P$. Interestingly, the set $\MM(P)$ always contains the trivial distance $\dist(x,y):={\bf 1}_{(x\ne y)}$, in which case the Wasserstein distance coincides with the total-variation distance $\dtv\left(\cdot,\cdot\right)$, so that $\kappa_\star(P)$ is well-defined and non-negative. Moreover, when specialized to this crude metric, our main result has the following immediate consequence.
\begin{corollary}\label{cor:dbarcor}
For any Markov semi-group $(P_t)_{t\ge 0}$ on any finite state space $\cX$, for any initial law $\mu\in\cP(\cX)$,  any time $t\ge 0$,  
\begin{eqnarray}
\label{universal}
H(\mu P_t\,|\,\pi) & \le & \overline{\rm d}(t)\,H(\mu\,|\,\pi),\quad \textrm{where}\quad 
\overline{\rm d}(t)\ :=\ \max_{x,y\in\cX}\dtv\left(P_t(x,\cdot),P_t(y,\cdot)\right).
\end{eqnarray}
\end{corollary}
We point out that the total-variation distance $\overline{d}$ appearing in (\ref{universal}) is a classical and well-studied quantity in mixing-time theory (see, e.g., \cite[Section 4.4]{MR3726904}). Yet and perhaps surprisingly, its role as a universal entropy dissipation factor seems to be new. 

\section{Proof of the main result}
In this section we prove Theorem \ref{th:main}. 
As in many other applications of curvature, we shall actually work with the dual Kantorovich-Rubinstein  formulation, which we now recall.  Let
\begin{eqnarray}
\lip(f) & := & \sup\left\{\frac{\left|f(x)-f(y)\right|}{\dist(x,y)}\colon(x,y)\in\cX^2,x\ne y\right\},
\end{eqnarray}
denote the Lipschitz constant of a function $f\colon\cX\to\dR$. The first inequality in the following result constitutes a well-known characterization of the Ollivier curvature $\cc(P)$. The second inequality explicitly appears in \cite[Theorem 4.3]{munch2023ollivier} as a characterization of non-negative sectional curvature in the special case where the underlying metric $\dist$ is the combinatorial distance. Interestingly, this characterization turns out to fail in the more general setup that we consider here (we found an explicit $4\times 4$ matrix $P$ satisfying  (\ref{sectional}) but not Assumption \ref{assume:sectional}). Nevertheless, the direct implication remains valid, and this is all we actually need. 
\begin{lemma}[Dual formulations of curvature and sectional curvature]\label{lm:lip}\ 
\begin{enumerate}[(i)]
\item For any $f\colon\cX\to\dR$, we have
\begin{eqnarray}
\label{curvature}
\lip(Pf) & \le & \left(1-\cc(P)\right)\lip(f).
\end{eqnarray}
\item 
Under Assumption (\ref{assume:sectional}), we also have for any $f\colon\cX\to(0,\infty)$, 
\begin{eqnarray}
\label{sectional}
\lip\left(\log P^\star f\right) & \le & \lip(\log f).
\end{eqnarray}
\end{enumerate}
\end{lemma}
\begin{proof}Fix a function $f\colon\cX\to\dR$ and two points $x,y\in\cX$. By definition, there is a coupling $(X,Y)$ of $P(x,\cdot)$ and $P(y,\cdot)$ such that $\EE[\dist(X,Y)]\le\left(1-\cc(P)\right)\dist(x,y)$. Since, $\EE[f(X)]=Pf(x)$ and $\EE[f(Y)]=Pf(y)$, we can then write
\begin{eqnarray*}
|Pf(x)-Pf(y)|  
& \le & \EE\left[|f(X)-f(Y)|\right]\\
& \le & \lip(f)\EE[\dist(X,Y)]\\
& \le & \left(1-\cc(P)\right)\lip(f)\dist(x,y).
\end{eqnarray*}
This  establishes the first claim. We now assume that $f$ is positive and that $(X^\star,Y^\star)$ is a  coupling of $P^\star(x,\cdot)$ and $P^\star(y,\cdot)$ such that $\dist(X^\star,Y^\star)\le\dist(x,y)$ almost-surely. Then, we have 
\begin{eqnarray*}
\log f(X^\star) & \le & \log f(Y^\star)+\dist(X^\star,Y^\star)\lip(\log f)\\
& \le &   \log f(Y^\star)+\dist(x,y)\lip(\log f). 
\end{eqnarray*} 
We now take exponentials, then expectations, and finally logarithms again to arrive at
\begin{eqnarray*}
\log P^\star f(x) & \le & \log P^\star f(y) + \dist(x,y)\lip(\log f). 
\end{eqnarray*} 
Since this is true for all $x,y\in\cX$, the second claim is proved. 
\end{proof}
We henceforth let $\alpha=\alpha(P)$ denote the optimal constant in the entropic contraction (\ref{EI}): 
\begin{eqnarray}
\label{def:alpha}
\alpha \ := \  1-\sup_{\mu\ne \pi}\cH(\mu), & \qquad\textrm{where}\qquad & \cH(\mu)\ :=\ \frac{H(\mu P\,|\,\pi)}{H(\mu\,|\,\pi)}.
\end{eqnarray}
Our starting point is the following simple observation about optimizers of $\cH$. 
\begin{lemma}\label{lm:optimizer}If a measure $\mu\in\cP(\cX)\setminus\{\pi\}$ maximizes $\cH$, then its density  $f:=\frac{\mu}{\pi}$ satisfies
\begin{eqnarray*}
(P \log P^\star f)(x) & = & (1-\alpha)\log f(x),
\end{eqnarray*}
for all $x\in\cX$, this identity being understood in $\dR\cup\{-\infty\}$.
\begin{proof}
Suppose that  $\mu\in\cP(\cX)\setminus\{\pi\}$ achieves the supremum of $\cH$, and write $f:=\frac{\mu}{\pi}$ for its density. Fix a point $x\in\cX$ and assume first that $\mu(x)>0$. Then the formula 
\begin{eqnarray}
\label{def:ftheta}
\mu_\theta & := & \frac{\mu+\theta \delta_x}{1+\theta},
\end{eqnarray}
defines a probability measure for all small enough $\theta\in\dR$, and an easy differentiation yields:
\begin{eqnarray}
\label{first}
\frac{1}{\theta}\left(H(\mu_\theta\,|\,\pi)-H(\mu\,|\,\pi)\right) & \xrightarrow[\theta\to 0]{} & \log f(x)-H(\mu\,|\,\pi);\\
\label{second}
\frac{1}{\theta}\left(H(\mu_\theta P\,|\,\pi)-H(\mu P\,|\,\pi)\right) & \xrightarrow[\theta\to 0]{} & (P\log P^\star f)(x)-H(\mu P\,|\,\pi).
\end{eqnarray}
Recalling that $\cH(\mu)=1-\alpha$, we easily deduce that 
\begin{eqnarray}
\label{taylor}
\frac{1}{\theta}\left(\cH(\mu_\theta) - \cH(\mu)\right) & \xrightarrow[\theta\to 0]{} & \frac{(P\log P^\star f)(x)-(1-\alpha)\log f(x)}{H(\mu\,|\,\pi)}.
\end{eqnarray}
Since $\cH$ is maximized at $\mu$, the right-hand side must vanish, yielding the desired identity. Let us now consider the degenerate case where $\mu(x)=0$. Then, in order for  (\ref{def:ftheta}) to define an element of $\cP(\cX)$, we need to restrict the parameter $\theta$ to non-negative values. Under this restriction, the convergences (\ref{first}) and (\ref{second}) hold in $\dR\cup\{-\infty\}$. If we had $(P\log P^\star f)(x)\in\dR$, then the convergence (\ref{taylor}) would still hold, but the limit would now be $+\infty$, contradicting the fact that $\cH$ is maximal at $\mu$. Thus, we must have $(P\log P^\star f)(x)=-\infty$, and the claimed identity holds with both sides being equal to $-\infty$.
\end{proof}
\end{lemma}
Our second ingredient is the following  result, which complements the above lemma by investigating the behavior of the functional $\mu\mapsto\cH(\mu)$ near the singularity point $\mu=\pi$. 
\begin{lemma}\label{lm:linear}Let $(\mu_n)_{n\ge 1}$ be elements of $\cP(\cX)\setminus\{\pi\}$ that converge to $\pi$. Then, 
\begin{eqnarray*}
\limsup_{n\to\infty}\cH(\mu_n) & \le & 1-\cc(PP^\star).
\end{eqnarray*}
\end{lemma}
\begin{proof}
Write $\mu_n=(1+h_n)\pi$, where $h_n\colon\cX\to\dR$ is a non-constant function with zero stationary mean, which vanishes as $n\to\infty$. Then, an easy use of the Taylor expansion $(1+\theta)\log(1+\theta)=\theta+\frac{1}{2}\theta^2+o(\theta^2)$ as $\theta\to 0$ gives the asymptotics
\begin{eqnarray*}
H(\mu_n\,|\,\pi) & \sim & \frac{1}{2}\|h_n\|^2;\\
H(\mu_n P\,|\,\pi) & \sim & \frac{1}{2}\left\|P^\star h_n\right\|^2,
\end{eqnarray*}
where $\|\cdot\|$ denotes the norm in the Hibert space $L^2(\cX,\pi)$, and where the notation $a_n\sim b_n$ means that $a_n/b_n\to 1$ as $n\to\infty$.
It follows that
\begin{eqnarray}
\label{taylor2}
\cH(\mu_n) & \sim & \frac{\left\|P^\star h_n\right\|^2}{\left\|h_n\right\|^2}.
\end{eqnarray}
Now, let $1=\lambda_1>\lambda_2\ge\ldots\ge \lambda_N\ge 0$ denote the  $N=|\cX|$ ordered eigenvalues of the non-negative self-adjoint operator $PP^\star$, and let $(\phi_1,\ldots,\phi_N)$ be a corresponding orthonormal eigenbasis with $\phi_1=1$. We can then write, for any $h\colon\cX\to\dR$,
\begin{eqnarray*}
\|P^\star h\|^2 \ = \ \langle h, PP^\star h\rangle & = & \sum_{i=1}^N\lambda_i\langle h,\phi_i\rangle^2,
\end{eqnarray*}
where $\langle \cdot,\cdot\rangle$ denotes the scalar product in $L^2(\cX,\pi)$.
Since $\langle h,\phi_1\rangle$ is exactly the stationary mean of $h$, we deduce that when $h$ is centered,
\begin{eqnarray*}
\|P^\star h\|^2 & = & \sum_{i=2}^N\lambda_i\langle h,\phi_i\rangle^2  \ 
 \le \ \lambda_2\sum_{i=2}^N\langle h,\phi_i\rangle^2 \ = \ \lambda_2\|h\|^2.
\end{eqnarray*}
This applies in particular to $h=h_n$, and inserting this into  (\ref{taylor2}) shows that 
\begin{eqnarray*}
\limsup_{n\to\infty}\cH(\mu_n) & \le & \lambda_2.
\end{eqnarray*}
To conclude, observe that $\lambda_2\le 1-\cc(PP^\star)$, as can be seen by choosing $f=\phi_2$ in the dual formulation of $\cc(PP^\star)$ (Lemma \ref{lm:lip} (i) applied to $PP^\star$ instead of $P$).
\end{proof}

We now have everything we need to prove our main result.
\begin{proof}[Proof of Theorem \ref{th:main}]Our goal is to prove that under Assumption \ref{assume:sectional},
\begin{eqnarray}
\label{goal}
\alpha  & \ge & \cc(P).
\end{eqnarray}
We first make the extra assumption that all entries of $P$ are positive.  By the very definition  (\ref{def:alpha}), there exists a sequence $(\mu_n)_{n\ge 1}$ in $\cP(\cX)\setminus\{\pi\}$ such that
\begin{eqnarray}
\label{supremum}
\cH(\mu_n) & \xrightarrow[n\to\infty]{} & 1-\alpha.
\end{eqnarray}
Since $\cX$ is finite,  we can safely assume -- upon extracting a subsequence if needed -- that $(\mu_n)_{n\ge 1}$ converges pointwise to a limit $\mu\in\cP(\cX)$. If $\mu=\pi$, then Lemma \ref{lm:linear} ensures that
\begin{eqnarray*}
1-\alpha & \le & 1-\cc(PP^\star)\\
& \le & \left(1-\cc(P)\right)\left(1-\cc(P^\star)\right).
\end{eqnarray*}
But Assumption \ref{assume:sectional} guarantees that $\cc(P^\star)\ge 0$, so (\ref{goal}) is proved. On the other hand, if $\mu\ne\pi$, then Lemma \ref{lm:optimizer} ensures that $f:=\frac{\mu}{\pi}$ must satisfy the functional equation
\begin{eqnarray}
\label{eq:maximizer}
P\log P^\star f & = & (1-\alpha)\log f.
\end{eqnarray}
Recall that this equality a priori holds in $\dR\cup\{-\infty\}$. However, since all entries of $P$ were assumed to be positive, the function $P^\star f$ is strictly positive. Thus, the left-hand side of (\ref{eq:maximizer}) is actually finite, and hence so is the right-hand side. In other words, $f$ is strictly positive, and we may therefore safely invoke Lemma \ref{lm:lip} to write
\begin{eqnarray*}
(1-\alpha)\lip(\log f) & = & \lip\left(P\log P^\star f\right)\\
& \le & \left(1-\cc(P)\right)\lip\left(\log P^\star f\right)\\
& \le & \left(1-\cc(P)\right)\lip(\log f).
\end{eqnarray*}
Since $f$ is non-constant ($\mu\ne\pi$), we may finally simplify through by $\lip(\log f)$ to obtain the desired conclusion. To handle the general case where some entries of $P$ may vanish, we introduce a perturbation parameter $\varepsilon\in(0,1)$ and replace each entry $P(x,y)$ with 
\begin{eqnarray*}
P_\varepsilon(x,y) & := & (1-\varepsilon)P(x,y)+\varepsilon\pi(y).
\end{eqnarray*}
The stochastic matrix $P_\varepsilon$ only has positive entries, and its stationary distribution is $\pi$. Moreover,  $\cc(P_\varepsilon)\ge\cc(P)$: indeed, given $x,y\in\cX$ and a coupling $(X,Y)$ of $P(x,\cdot)$ and $P(y,\cdot)$, we can construct a coupling $(X_\varepsilon,Y_\varepsilon)$ of $P_\varepsilon (x,\cdot)$ and $P_\varepsilon(y,\cdot)$ such that $\dist(X_\varepsilon,Y_\varepsilon)\le\dist(X,Y)$ by generating an independent pair $(Z,B)$ with $Z\sim\pi$ and $B\sim\textrm{Bernoulli}(\varepsilon)$ and setting
\begin{eqnarray*}
(X_\varepsilon,Y_\varepsilon) & := & 
\left\{
\begin{array}{ll}
(X,Y) & \textrm{if }B=0\\
(Z,Z) & \textrm{if }B=1.
\end{array}
\right.
\end{eqnarray*} 
The same argument  applies to the adjoint $P_\varepsilon^\star$ and shows that the latter inherits Assumption \ref{assume:sectional} from $P^\star$. Thus, the first part of our proof applies to the perturbed matrix $P_\varepsilon$ and allows us to conclude that $\alpha(P_\varepsilon)\ge\cc(P)$. This means that for each $\mu\in\cP(\cX)$, we have
\begin{eqnarray*}
H(\mu P_\varepsilon\,|\,\pi)& \le & \left(1-\cc(P)\right)H(\mu\,|\,\pi).
\end{eqnarray*}
We may finally send $\varepsilon\to 0$ to conclude.
\end{proof}
\section{Applications}
\label{sec:applications}
In this final section, we illustrate the strength of Theorem \ref{th:main} by establishing new entropy dissipation estimates for several important classes of Markov chains. 

\subsection{Birth and Death Processes}
We first take a look at the  case of Birth and Death Processes (BDP). Specifically, we set $\cX:=\{1,\ldots,n\}$ and consider the generator that acts on any function $f\colon\cX\to\dR$ as follows:
\begin{eqnarray}
\label{def:BDC}
Lf(x) & := & q_+(x)\left(f(x+1)-f(x)\right)+q_-(x)\left(f(x-1)-f(x)\right),
\end{eqnarray}
where $q_\pm$ are arbitrary positive functions on $\cX$, except that $q_-(1)=q_+(n)=0$. This generator is reversible with respect to the probability measure
\begin{eqnarray*}
\forall x\in\cX,\qquad \pi(x) & := & \frac{1}{C}\prod_{k=2}^{x}\frac{q_+(k-1)}{q_-(k)},
\end{eqnarray*}
where $C$ is a normalizing constant.
Let us now assume the following monotonicity:
\begin{eqnarray}
\label{assume:BDC}
\forall x\in\{1,\ldots,n-1\},\qquad q_+(x+1)\le q_+(x) & \textrm{ and } & q_-(x+1)\ge q_-(x).
\end{eqnarray}
This condition easily guarantees that  our BDP $(X_t)_{t\ge 0}$ starting from any $X_0\in\{1,\ldots,n-1\}$ can be coupled with a  BDP $(Y_t)_{t\ge 0}$ starting from $Y_0=X_0+1$ so that
\begin{eqnarray}
\label{BDC:sectional}
\forall t\ge 0,\qquad Y_t-X_t & \in & \{0,1\}.
\end{eqnarray}
This already shows that the underlying semi-group $(P_t)_{t\ge 0}$ has non-negative sectional curvature with respect to the metric $\dist\colon (x,y)\mapsto|y-x|$, which is generated by pairs of consecutive states. Moreover, the same coupling yields for all $t\ge 0$,
\begin{eqnarray}
1-\cc(P_t) & \le & \EE[\dist(X_t,Y_t)] \ = \ \EE[Y_t]-\EE[X_t].
\end{eqnarray}
Using the notation $\EE_x[\cdot]$ to indicate that the initial state is $x\in\cX$, we obtain:
\begin{corollary}\label{co:BDC}Under the condition (\ref{assume:BDC}), the semi-group $(P_t)_{t\ge 0}$ generated by (\ref{def:BDC}) satisfies
\begin{eqnarray*}
H(\mu P_t\,|\,\pi) & \le & {\mm}(t)\,H(\mu\,|\,\pi),\qquad \textrm{ where }\qquad \mm(t)\ :=\  \max_{1\le x<n}\left(\EE_{x+1}[X_t]-\EE_x[X_t]\right),
\end{eqnarray*}
for every initial law $\mu\in\cP(\cX)$ and every time $t\ge 0$.
\end{corollary}To appreciate this result, let us give two simple and generic bounds on the function $t\mapsto\mm(t)$. The first one, obtained by an easy Grönwall argument, is
\begin{eqnarray*}
\mm(t) & \le & e^{-\delta t}, \qquad\textrm{where}\qquad \delta \ := \ \min_{1\le x<n}\left\{q_+(x)-q_+(x+1)+q_-(x+1)-q_-(x)\right\}.
\end{eqnarray*}
Inserting this into Corollary \ref{co:BDC} readily yields a MLSI with constant $\delta$, which is an important classical  result  \cite[Theorem 3.1]{MR2548501}. However, our time-varying estimate has the advantage of being meaningful even when $\delta=0$: for example,  (\ref{BDC:sectional}) ensures that $X$ and $Y$ must have met by the time at which $X$ hits $n$ or $Y$ hits $1$, yielding the alternative bound
\begin{eqnarray}
\mm(t) & \le & \PP_1(T_n>t)\wedge \PP_n(T_1>t),
\end{eqnarray}
where $T_z:=\min\{t\ge 0\colon X_t=z\}$ denotes the hitting time of $z$.
 As a concrete example, consider the extreme case where the jump rates are all equal to $1$: this corresponds to simple random walk on the segment, for which we classically have $\EE_1[T_n]=\Theta(n^2)$. Thus, Corollary \ref{co:BDC} shows that the entropy decay occurs on a time-scale of order $n^2$, which is actually sharp.

\subsection{Colored Exclusion Processes}
\label{sec:coloredex}
In this section, we consider a non-conservative and colored version of the popular Exclusion Process. The model is parametrized by the following ingredients:
\begin{itemize}
\item a finite set $\cS$ (the colors) equipped with a fully supported probability law $\nu$;
\item an integer $n\in\dN$ (the dimension);
\item a non-negative symmetric array $\{c(i,j)\colon {1\le i\ne j\le n}\}$ (the exchange rates);
\item a non-negative vector $\{r(i)\colon 1\le i\le n\}$ (the refresh rates).
\end{itemize}
By definition, the Colored Exclusion Process (CEP) with those parameters is the continuous-time Markov chain on $\cX=\cS^n$ whose generator acts on any function $f\colon\cX\to\dR$ as follows:
\begin{eqnarray}
\label{def:CEP}
Lf(x) & := & \sum_{1\le i<j\le n}c(i,j)\left(f(x^{i \leftrightarrow j})-f(x)\right)+\sum_{1\le i\le n}r(i)\sum_{\sigma\in\cS}\nu(\sigma)\left(f(x^{i,\sigma})-f(x)\right),
\end{eqnarray}
where $x^{i \leftrightarrow j}$ (resp. $x^{i,\sigma}$) denotes the configuration obtained from $x$ by swapping the $i-$th and $j-$th entries (resp. replacing the $i-$th entry with $\sigma$). In more concrete terms, each pair of sites $\{i,j\}$ exchange values at rate $c(i,j)$, and each site $i$ resamples its color afresh according to the law $\nu$ at rate $r(i)$. This dynamics is clearly reversible w.r.t. the product law
\begin{eqnarray*}
\pi(\dist x) & := & \prod_{i=1}^n\nu(\dist x_i).
\end{eqnarray*}
Moreover, it is irreducible as soon as the support of $r$ intersects each connected component of the graph induced by the support of $c$, which we henceforth assume. It was shown in \cite{MR4546624} that the mixing properties of the $|\cS|^n-$dimensional generator (\ref{def:CEP}) are intimately related to those of the much simpler $n\times n$ Laplace matrix 
\begin{eqnarray}
\Delta(i,j) & := & 
\left\{
\begin{array}{ll}
c(i,j) & \textrm{ if }i\ne j\\
-r(i)-\sum_{k\ne i} c(i,k) & \textrm{ if }i=j.
\end{array}
\right.
\end{eqnarray}
This symmetric matrix  describes the evolution of a random walk on $[n]$ which jumps according to the conductances $c(\cdot, \cdot)$ and is killed at the space-varying rate $r(\cdot)$. We let $t\mapsto \PP_i(T>t)$ denote the tail distribution function of the life-time of such a killed random walk, when started from site $i\in[n]$. Note that we have the spectral representation
\begin{eqnarray*}
\PP_i(T>t) & = & \sum_{k=1}^n e^{-\lambda_k t}\phi_k(i)\langle \phi_k,1\rangle,
\end{eqnarray*}
where $\langle a,b\rangle=\sum_ia(i)b(i)$ is the standard scalar product,   $\lambda_n\ge \ldots\ge \lambda_1> 0$ denote the eigenvalues of $-\Delta$, and $\phi_1,\ldots,\phi_n$ is a corresponding orthonormal basis of eigenvectors. Explicit estimates are available in many concrete examples (see \cite{MR4546624}). Our main theorem provides the following entropy contraction principle.
\begin{corollary}\label{co:CEP}For any $\mu\in\cP(\cX)$ and any  $t\ge 0$, the transition matrix $P_t:=e^{tL}$ satisfies 
\begin{eqnarray*}
H(\mu P_t\,|\,\pi) & \le & H(\mu\,|\,\pi)\max_{1\le i\le n}\PP_i(T>t).
\end{eqnarray*}
\end{corollary}
\begin{proof}We equip $\cX$ with the \emph{Hamming distance} $\dist(x,y):=\#\{i\in[n]\colon x_i\ne y_i\}$. Note that this is \emph{not} the combinatorial distance associated with $L$, unless $r(\cdot)$ has full support. Now, fix an initial pair $(X_0,Y_0)\in\cX^2$ with $\dist(X_0,Y_0)=1$ (such pairs clearly generate the above metric) and consider the Markov chain $(X_t,Y_t)_{t\ge 0}$ on $\cX^2$ that evolves as follows:
\begin{itemize}
\item a joint exchange $(x,y)\to(x^{i\leftrightarrow j},y^{i\leftrightarrow j})$ occurs at rate $c(i,j)$ for each $1\le i<j\le n$.
\item a joint refresh $(x,y)\to(x^{i,\sigma},y^{i,\sigma})$ occurs at rate $r(i)\nu(\sigma)$ for each $(i,\sigma)\in[n]\times\cS$.
\end{itemize}
It is clear that $X$ and $Y$ are then distributed as CEPs. Moreover, the distance $(x,y)\mapsto\dist(x,y)$ is preserved or reduced under each jump, so that $\dist(X_t,Y_t)\le 1$ for all $t\ge 0$. This  establishes non-negative sectional curvature along the semi-group $(P_t)_{t\ge 0}$. In fact, the same coupling also provides an  estimate on the curvature. Specifically, the coordinate at which $X_t$ and $Y_t$ differ evolves exactly as a killed random walk with generator $\Delta$, so that
\begin{eqnarray*}
\forall t\ge 0,\qquad 1-\cc(P_t) & \le & \max_{i\in[n]}\PP_i(T> t),
\end{eqnarray*}
where $T$ is the life-time of the walk. 
Applying Theorem \ref{th:main} to $P_t$ concludes the proof.
\end{proof}
\subsection{Generalized Interchange Processes}
We now turn to a very general class of random walks on the symmetric group, which contains in particular the well studied Interchange Process. Specifically, given an integer $n\in\dN$ and a function  $c\colon 2^{[n]}\to [0,\infty)$, we consider the continuous-time  Markov chain on $\cX=\fS_n$, the symmetric group of permutations of $[n]$, whose generator acts as follows: for any function $f\colon\cX\to\dR$ and any state $x\in\cX$, 
\begin{eqnarray}
\label{def:GIP}
Lf(x) & = & \sum_{A\subseteq [n]}\frac{c(A)}{|A|!}\sum_{\sigma\in\mathfrak S_A}\left(f(x\sigma)-f(x)\right),
\end{eqnarray}
where $\fS_A$ denotes the group of permutations on $A$, and $|A|$ is the cardinality of the subset $A$. This generator is clearly reversible w.r.t the uniform law $\pi$ on $\cX$. We may think of a permutation $x\in\cX$ as assigning a unique label $x_i\in[n]$ to each site $i\in[n]$. The dynamics (\ref{def:GIP}) then simply shuffles the labels of all sites $i\in A$ uniformly at random at rate $c(A)$, for each  block $A\subseteq[n]$. 
When the rate function $A\mapsto c(A)$ is supported on blocks of size $2$,  which one interprets as weights on the edges of a graph, the process is known as the
Interchange Process. Thus, the general case is viewed as an Interchange Process on a weighted hypergraph.
We observe that the individual motion of each label is just a continuous-time random walk on $[n]$ with conductances
\begin{eqnarray}
\label{single}
\widehat{c}(i,j) & := & \sum_{A\supseteq\{i,j\}}\frac{c(A)}{|A|}.
\end{eqnarray} 
Relating the mixing properties of the high-dimensional process generated by (\ref{def:GIP}) to those of its one-dimensional marginals (\ref{single}) is a natural and important problem, which has been the subject of active research  \cite{MR2629990,dieker2010interlacings,oliveira2013mixing,
 MR3413687,MR4164852,MR4254474,bristiel2022entropy,alon2023aldous}. The first author conjectured that, for any choice of weights $c$, the spectral gap of the process generated by \eqref{def:GIP} coincides with the spectral gap of the single particle process with  rates \eqref{single},  see \cite[Conjecture 1.7]{bristiel2022entropy}. This represents the hypergraph generalization of the renowned Aldous conjecture, affirming the aforementioned equivalence for the Interchange Process. While Aldous conjecture was proved in \cite{MR2629990}, the general case has been verified only for certain classes of hypergraph weights, see \cite{bristiel2022entropy,alon2023aldous}. 
  We note that for a given set of edge conductances $\widehat c$, there may exist multiple choices of hypergraph weights $c$ that satisfy \eqref{single}. The conjecture thus asserts that all such choices yield the same spectral gap. It is noteworthy that our estimate below, which is a simple application of our main result,  provides control over the decay of relative entropy, independent of the specific choice of $c$ compatible with \eqref{single}, offering additional support for the conjecture's validity.
%
\begin{corollary}\label{corshuffle}
For any   $\mu\in\cP(\cX)$ and   $t\ge 0$, the matrix $P_t:=e^{tL}$ generated by (\ref{def:GIP}) satisfies
\begin{eqnarray}
\label{GIP:decays}
H(\mu P_t\,|\,\pi) & \le & H(\mu\,|\,\pi)\max_{1\le i,j\le n}\PP_{i,j}\left(T>t\right),
\end{eqnarray}
where $T:=\min\{t\ge 0\colon I_t=J_t\}$  denotes the meeting time of two independent random walks on $[n]$ with conductances (\ref{single}) starting from $I_0=i$ and $J_0=j$, respectively. 
\end{corollary}
\begin{proof}
We equip $\cX$ with the \emph{transposition distance} $\dist(x,y)$, which is the minimal number of swaps that need to be performed in order to turn $x$ into $y$. This metric is generated, in the sense of (\ref{geodesic}), by those pairs of states that differ in exactly two coordinates. Starting from such a pair $(X_0,Y_0)$, we can construct a coupling $(X_t,Y_t)_{t\ge 0}$ as follows: we equip each block $A\subseteq[n]$ with an independent Poisson clock of rate $c(A)$ and, whenever the clock rings, we simply replace the current state $(x,y)$ with $(x\sigma,y\sigma)$, where $\sigma$ is a uniformly chosen element of $\fS_A$. Such a transformation clearly preserves the distance, and this already establishes non-negative sectional curvature along the semi-group. Now, let us modify our coupling as follows: whenever the clock of a block $A$ rings, if the current states $x$ and $y$ happen to agree outside $A$, we ensure coalescence by replacing $(x,y)$ with $(x\sigma,x\sigma)$ instead of $(x\sigma,y\sigma)$. It is easy to see that the time at which this occurs is stochastically dominated by the meeting time $T$ of two independent random walks with conductances (\ref{single}) starting at the two sites where $X_0$ and $Y_0$ differ. This yields the curvature estimate 
\begin{eqnarray*}
1-\cc(P_t) & \le & \max_{1\le i,j\le n}\PP_{i,j}(T>t).
\end{eqnarray*}
Applying Theorem \ref{th:main} to $P_t$ concludes the proof. 
\end{proof}
To give a concrete example, consider the  unit rate Interchange Process on a segment of length $n$, which corresponds to the rate function
\begin{eqnarray*}
c(A) & := & \left\{
\begin{array}{ll}
2 & \textrm{if }A=\{i,i+1\}\textrm{ for some }i\in[n-1]\\
0 & \textrm{else}.
\end{array}
\right.
\end{eqnarray*}
Then the one-dimensional dynamics (\ref{single}) is that of the simple random walk on the $n-$segment, for which it is classical that the worst-case meeting time  $T$ is of order $n^2$. 
Thus, (\ref{GIP:decays}) shows that the  entropy decay occurs on a time-scale of order $n^2$, which is actually sharp. It is also instructive to take a look at the Interchange Process on the complete graph, a.k.a.\ Random Transpositions, namely
\begin{eqnarray*}
c(A) & := & \left\{
\begin{array}{ll}
\frac{4}{n(n-1)} & \textrm{if }A=\{i,j\}\textrm{ for some }1\le i<j\le n\\
0 & \textrm{else}.
\end{array}
\right.
\end{eqnarray*}
In this  case, coalescence clearly occurs at the constant rate $\kappa=4/n(n-1)$, thereby providing the bound   (\ref{MLSI}) for the MLSI. The resulting mixing-time bound is rather poor, being off by a factor $n$ with respect to the known $O(n\log n)$ behavior \cite{diaconis1981generating}, but it can be considerably enhanced by using our time-varying estimate  (\ref{timevarying}). Indeed,  it was shown in \cite{MR3936154} that the curvature $t\mapsto \cc(P_t)$ undergoes a remarkable transition from $o(1)$ to $\Theta(1)$ as $t$ passes the critical value $n/2$, thereby providing an excellent illustration of the  ``\emph{curvature improves at large scales}'' principle.  In fact, it was shown in  \cite{MR3936154} that when $t/n \to\infty$ as $n\to\infty$, 
\begin{eqnarray*}
1-\cc(P_{t})& \le & e^{-\frac{(2-o(1))t}n}.
\end{eqnarray*}
Combining this with our main result, we deduce that the worst-case relative entropy to equilibrium $\sup_{\mu\in\cP(\cX)} H(\mu P_t\,|\,\pi)$ is $o(1)$ already by time $t=(n\log n)/(2-o(1))$. This estimate is sharp, and establishes  cutoff in relative entropy. Moreover, it provides an improvement over the best known bounds \cite{gao2003exponential,goel2004modified} which predict $e^{-t/(n-1)}$ for the relative entropy decay in this model. More generally, the same argument applies to the random walk generated by $k-$cycles for any $k=o(n)$ and allows us to conclude that the total-variation cutoff established in \cite{MR3936154} also occurs in relative entropy, which seems to be new. 

Another remark concerning the general result in Corollary \ref{corshuffle}
is that the exact same bound applies, for any fixed $k\in[n]$,  to the case where we have $k$ indistinguishable particles undergoing the same dynamics, that is a hypergraph version of the exclusion process with $k$ particles. This is obtained from \eqref{GIP:decays} by a simple projection argument, by declaring black all particles labeled $1,\dots,k$ and white all particles labeled $k+1,\dots,n$, and by keeping track only of the particle colors. The resulting bounds may be used to investigate mixing time relations in the spirit of \cite{oliveira2013mixing,connor2019mixing}.   
The same of course applies as well to the case where one has more than two colors, providing a conservative  version of the model with sources that we discussed in Section \ref{sec:coloredex}.

Finally, we observe that an estimate as in Corollary \ref{corshuffle}
 can be obtained for a slightly different model, where $k$ labeled walkers undergo synchronous updates along the hyperedges of a weighted hypergraph with rates given by the weight function $c$, but are otherwise independent. More precisely, one starts with the $k$ particles in arbitrary locations (with no constraints on their overlap), and the dynamics proceeds by synchronous updates  with rate $c(A)$ of all particles sitting at the vertices of the hyperedge $A\subset [n]$. The result of one update at $A$ is that all particles involved are independently reshuffled along the vertices in $A$. 
Thus the stationary measure is uniform over $[n]^k$.  This model was introduced in \cite{bristiel2022entropy}, where it was shown that the entropy decay of the system is controlled by the entropy decay of a single particle. The class of models includes in particular the so-called binomial splitting process studied in \cite{quattropani2023mixing}. An application of our main result here produces the exact same bound as in Corollary \ref{corshuffle}, for any fixed $k\in\dN$.
\subsection{Glauber Dynamics}

The celebrated \emph{Markov chain Monte Carlo revolution} in computational statistics is fundamentally based on the simple but far-reaching idea -- attributed to Metropolis \cite{doi:10.1063/1.1699114} and Hastings \cite{MR3363437} -- that approximate samples from a target probability distribution $\pi$ can be efficiently produced by running an appropriate Markov chain that admits $\pi$ as its equilibrium law; see the survey paper by P. Diaconis \cite{MR2476411} and the references therein. Among the various particular implementations that have been proposed, one of the most popular is probably \emph{Gibbs sampling}, also known as \emph{Glauber dynamics}. Sticking to our discrete setting for simplicity, let us assume that out target probability measure $\pi$ lives on  $\cS^n$, where $\cS$ is a fixed finite set. As before, we let $x^{i,\sigma}:=(x_1,\ldots,x_{i-1},\sigma,x_{i+1},\ldots,x_n)$ denote the vector obtained from $x$ by changing the $i-$th coordinate to $\sigma\in\cS$. We write $\cX:=\{x\in\cS^n\colon\pi(x)>0\}$ for the support of $\pi$ and, for each $x\in\cX$ and each $i\in[n]$, we let $\pi_i(\cdot|x)$ denote the conditional law of the $i-$th coordinate, given that the remaining coordinates agree with $x$: 
\begin{eqnarray*}
\pi_i(\sigma|x) & := & \frac{\pi(x^{i,\sigma})}{\sum_{\sigma'\in\cS}\pi(x^{i,\sigma'})}.
\end{eqnarray*}
The Glauber dynamics for $\pi$ is the Markov chain with state space $\cX$ and transition matrix
\begin{eqnarray}
\label{def:Glauber}
P(x,y) & := & \frac{1}{n}\sum_{i\in[n]}\sum_{\sigma\in\cS}\pi_i(\sigma|x){\bf 1}_{(y=x^{i,\sigma})}.
\end{eqnarray} In words, a transition consists in selecting one of the $n$ coordinates uniformly at random and resampling its content according to the measure $\pi$, conditioned on the current values of all other coordinates. This dynamics is  clearly reversible with respect to $\pi$, and irreducible as soon as the support of $\pi$ is connected under single-coordinate changes, which we henceforth assume.  In the idealized case where $\pi$ is a product measure, the entropy contraction (\ref{EI}) trivially holds with the optimal constant $\kappa=1/n$. In light of this, it is natural to hope for a similar behavior when the target distribution $\pi$ has \emph{weak dependencies}. The following general result formalizes this intuition.
\begin{corollary}\label{co:Glauber}Suppose that $\pi$ satisfies the following weak dependency condition:
\begin{eqnarray}
\label{assume:weakdep}
\pi_i(y_i|x)  & \ge & \sum_{j\ne i}\sum_{\sigma\ne x_j}\left(\pi_j(\sigma|y)-\pi_j(\sigma|x)\right)_+,
\end{eqnarray}
for all $i\in[n]$ and all $x,y\in\cX$ that differ exactly at the $i-$th coordinate. Then the matrix $(\ref{def:Glauber})$ has non-negative sectional curvature and satisfies the entropic contraction (\ref{EI}) with constant
\begin{eqnarray}
\label{def:kappa}
\kappa & := & \frac{1}{n}\min_{i,x,y}\left\{1-\sum_{j\ne i}\sum_{\sigma\ne x_j}\left|\pi_j(\sigma|y)-\pi_j(\sigma|x)\right|\right\} \ \ge 0,
\end{eqnarray}
where the minimum ranges over all $i\in[n]$ and all $x,y\in\cX$ that differ exactly at $i$.
\end{corollary}
\begin{proof}We equip $\cX$ with the combinatorial distance $\dist(x,y):=\min\{k\in\dN\colon P^k(x,y)>0\}$, which is generated by those pairs $(x,y)\in\cX^2$ that differ at a single coordinate $i\in[n]$. Now fix such a pair, and consider the coupling $(X,Y)$ of $P(x,\cdot)$ and $P(y,\cdot)$ defined as follows:
\begin{eqnarray*}
\left(X,Y\right) & := & \left\{
\begin{array}{ll}
(x,x) & \textrm{w.p. } {\frac{1}{n}\left(\pi_i(x_i|y)-\sum_{j\ne i}\sum_{\sigma\ne x_j}\left(\pi_j(\sigma|x)-\pi_j(\sigma|y)\right)_+\right)};\\
(y,y) & \textrm{w.p. } {\frac{1}{n}\left(\pi_i(y_i|x)-\sum_{j\ne i}\sum_{\sigma\ne x_j}\left(\pi_j(\sigma|y)-\pi_j(\sigma|x)\right)_+\right)};\\
(x^{i,\sigma},x^{i,\sigma}) & \textrm{w.p. } {\frac{1}{n}\pi_i(\sigma|x)}, \textrm{ for }\sigma\notin\{x_i,y_i\};\\
(x^{j,\sigma},y^{j,\sigma}) & \textrm{w.p. } {\frac{1}{n}\left(\pi_j(\sigma|x)\wedge \pi_j(\sigma|y)\right)}, \textrm{ for }j\ne i\textrm{ and }\sigma\ne x_j;\\
(x^{j,\sigma},x) & \textrm{w.p. } {\frac{1}{n}\left(\pi_j(\sigma|x)-\pi_j(\sigma|y)\right)_+}, \textrm{ for  }j\ne i\textrm{ and }\sigma\ne x_j;\\
(y,y^{j,\sigma}) & \textrm{w.p. } {\frac{1}{n}\left(\pi_j(\sigma|y)-\pi_j(\sigma|x)\right)_+}, \textrm{ for  }j\ne i\textrm{ and }\sigma\ne x_j;\\
(x,y) & \textrm{with the remaining probability}.
\end{array}
\right.
\end{eqnarray*}
Note that the first two probabilities are non-negative thanks to (\ref{assume:weakdep}). It is immediate to check that $X$ has law $P(x,\cdot)$ and that $Y$ has law $P(y,\cdot)$. Moreover, we have $\dist(X,Y)=0$ in the first three cases, and $\dist(X,Y)=1$ in the remaining cases. This shows that $P$ has non-negative sectional curvature. Finally, adding up the probabilities of the three first cases gives exactly the constant $\kappa$ appearing at (\ref{def:kappa}), and hence
\begin{eqnarray*}
\kappa(P) & \ge & 1-\EE[\dist(X,Y)] \ = \ \kappa.
\end{eqnarray*}
Applying our main result concludes the proof.
\end{proof}
Corollary \ref{co:Glauber} contains many special cases of interest, obtained by specializing the target measure $\pi$ to various popular spin systems such as the Ising Model, the Potts Model, or the Hard-Core Model (see the lecture notes \cite{MR1746301} for an introduction to those models). 
In all those examples and many others, our weak dependency assumption (\ref{assume:weakdep}) holds as soon as the temperature   is above an explicit threshold, and Corollary \ref{co:Glauber} guarantees entropy contraction with a constant of the right order of magnitude $\kappa=\Theta(1/n)$. Results of this form have a long history, and have been obtained using a variety of sophisticated methods  \cite{Zeg,Lu-Yau,MR1746301,Cesi,MR3434252,MR4015662,MR4415182,conforti2022probabilistic,BB19,chen2022localization}. Rather than delving into the specificity of each model and trying to optimize the associated constants, let us state one simple general consequence of  Corollary \ref{co:Glauber}  that applies to all high-temperature spin systems with pairwise interactions. 
While the result stated in Corollary \ref{cor:hightemp} below does not necessarily improve over existing bounds, it offers a comparable estimate with a simple and entirely different approach, thus providing an instructive application of our main result Theorem \ref{th:main}.
Specifically, consider a target measure $\pi$ of the form
\begin{eqnarray}
\label{def:spin}
\pi(x) & := & \frac{1}{C}\exp\left\{\sum_{1\le i< j\le n}\psi_{ij}(x_i,x_j)\right\},
\end{eqnarray}
where each $\psi_{ij}\colon\cS^2\to\dR$ is an arbitrary interaction function, and $C$ a normalizing constant. For convenience, set $\psi_{ij}(\sigma,\tau):=\psi_{ji}(\tau,\sigma)$ for $i>j$, and $\psi_{ii}(\tau,\sigma):=0$ for all $i$ and $(\tau,\sigma)\in\cS^2$. The influence of $i$ on $j$ is naturally measured by the quantity
\begin{eqnarray}
J_{ij} & := & \frac{1}{2}\max_{(\sigma,\sigma',\tau)\in\cS^2}|\psi_{ij}(\sigma,\tau)-\psi_{ij}(\sigma',\tau)|.
\end{eqnarray}
We can then define the maximal influence in our system as follows:
\begin{eqnarray}
\|J\| & := & (|\cS|-1)\max_{1\le i\le n}\left\{\sum_{j=1}^nJ_{ij}\right\}\,.
\end{eqnarray}
\begin{corollary}\label{cor:hightemp}
There is a universal constant $\varepsilon\in(0,1)$ ($\varepsilon=1/3$ works)  such that whenever 
$
\|J\|  \le \varepsilon,
$
the Glauber dynamics for (\ref{def:spin}) exhibits entropic contraction with constant 
\begin{eqnarray*}
\kappa & := & \frac{1-\|J\|}{n}\,.
\end{eqnarray*}
\end{corollary}
\begin{proof}
It follows from the definition that for any $x\in\cX$, $j\in[n]$, and $\sigma\in\cS$,
\begin{eqnarray}
\label{conditional}
\pi_j(\sigma|x) & = & \frac{1}{1+\sum_{\sigma'\ne\sigma}e^{\sum_{k}\psi_{jk}(\sigma',x_k)-\psi_{jk}(\sigma,x_k)}}\,.
\end{eqnarray}
Now, an elementary differentiation shows that for any $d\in\dN$ and any coefficients $a_1,\ldots,a_d\ge 0$, the function $F\colon\dR^d\to[0,1]$ defined by $F(u)=\frac{1}{1+a_1e^{u_1}+\cdots+a_de^{u_d}}$ satisfies $|F(u)-F(v)|\le \frac{1}{4}\|u-v\|_\infty$.
Consequently, if $x,y\in\cX$ differ exactly at the $i-$th coordinate, we obtain 
\begin{eqnarray*}
\left|\pi_j(\sigma|x)-\pi_j(\sigma|y)\right| & \le & \frac{1}{4}\max_{\sigma'\ne\sigma}\left\{\psi_{ij}(x_i,\sigma')-\psi_{ij}(y_i,\sigma')+\psi_{ij}(y_i,\sigma)-\psi_{ij}(x_i,\sigma)\right\} \
 \le \ J_{ij}\,.
\end{eqnarray*}
Thus, the right-hand side of (\ref{assume:weakdep}) is at most $\|J\|$. On the other hand, in view of the expression (\ref{conditional}), the left-hand side is at least $\left(1+qe^{\frac{2\|J\|}{q}}\right)^{-1}$, where $q=|\cS|-1$. It follows that (\ref{assume:weakdep}) is satisfied as soon as $\|J\|\le \varepsilon_q$, where $\varepsilon_q\in(0,1)$ is the unique solution to the equation
\begin{eqnarray*}
{\varepsilon_q} & = & \left(1+e^{\frac{2\varepsilon_q}{q}}\right)^{-1}.
\end{eqnarray*}
Note that $\varepsilon_q$ increases with $q$, so that $\varepsilon_q\ge \varepsilon_1\approx 0.337$. Finally, the bound $\left|\pi_j(\sigma|x)-\pi_j(\sigma|y)\right|\le J_{ij}$ shows that the constant $\kappa$ in Corollary \ref{co:Glauber} is at least $\frac{1-\|J\|}{n}$, and the proof is complete.
\end{proof}
\subsection{Zero-Range Processes}
\label{sec:ZRP}
Introduced by Spitzer \cite{Spitzer}, the \emph{Zero-Range Process} (\textsc{ZRP}) is a  generic interacting particle system  in which individual jumps occur at a rate that only depends on the current number of particles present at the source. The model is parameterized by the following ingredients:
\begin{itemize}
\item two  integers $m,n\ge 1$ representing the number of particles and sites, respectively;
\item an irreducible stochastic matrix $G=(G_{ij})_{1\le i,j\le n}$  specifying the geometry;
\item a function $r_i\colon\{1,2,\ldots\}\to(0,\infty)$ encoding the kinetics at each site $i\in [n]$.
\end{itemize}
The \textsc{ZRP} with these parameters is a continuous-time Markov chain with state space 
\begin{eqnarray}
\label{def:ZRPX}
\cX & := & \left\{(x_1,\ldots,x_n)\in\dZ_+^n\colon \sum_{i=1}^n x_i=m\right\},
\end{eqnarray}
and generator $L$ acting as follows: for any $f\colon\cX\to\dR$ and any $x=(x_1,\ldots,x_n)\in\cX$, 
\begin{eqnarray}
\label{def:ZRPL}
(L f)(x) & := & \sum_{1\le i,j\le n} r_i(x_i)G_{ij}\left(f(x+\delta_j-\delta_i)-f(x)\right),
\end{eqnarray}
where $(\delta_1,\ldots,\delta_n)$ denotes the canonical $n-$dimensional basis, and with the convention that $r_i(0)=0$ for all $i\in [n]$ (no jumps from empty sites). In words, a site $i$ with $k$ particles expels a particle at rate $r_i(k)$, and the latter goes to site $j$ with probability $G_{ij}$. It is immediate to check that the generator $L$ is irreducible, with invariant measure
\begin{eqnarray}
\pi(x) & := & \frac{1}{C}\prod_{i=1}^n\frac{\nu_i^{x_i}}{r_i(1)r_i(2)\cdots r_i(x_i)},
\end{eqnarray}
where $\nu=\nu_G$ denotes the unique invariant law of $G$ and where $C$ is a normalizing constant. Let us point out that $L$ is \emph{not} reversible, unless $G$ is. More precisely, the adjoint $L^\star$  is obtained from $L$  by  replacing the matrix $G$ with its adjoint $G^\star$ in the formula (\ref{def:ZRPL}). We henceforth make the standard assumption that the rate functions are monotone:
\begin{eqnarray}
\label{assume:monotone}
\forall i\in[n],\qquad \forall k\in[m],\qquad r_i(k) & \ge & r_i(k-1).
\end{eqnarray}
Also, we  equip the state space $\cX$ with (half) the $L^1$ distance $\dist(x,y) := \frac 12\sum_{i=1}^n|x_i-y_i|$.  Notice, again, that this is \emph{not}  the combinatorial distance induced by $L$,  unless $G$ has full support. Nevertheless, this choice is the `right' one in view of the following  result. 
\begin{lemma}[Attractiveness]\label{lm:ZRP}Under assumption (\ref{assume:monotone}), both the transition matrix $P_t:=e^{tL}$ and its adjoint $P_t^\star:=e^{tL^\star}$  have non-negative sectional curvature at any time $t\ge 0$.  
\end{lemma}
\begin{proof} Fix $z=(z_1,\ldots,z_n)\in\dZ_+^n$ with $\sum_{i=1}^nz_i=m-1$, and  $i\in[n]$. Now, let $Z=(Z_t)_{t\ge 0}$ be a ZRP with $m-1$ particles starting from $Z_0=z$ and, conditionally on $Z$, let $I=(I_t)_{t\ge 0}$ be a time-inhomogeneous random walk on $[n]$ starting from $I_0=i$ and jumping from any site $u$ to any site $v$ at the time-varying rate $\left[r_u(Z_t+1)-r_u(Z_t)\right]\times G_{uv}$. Then, the formula
\begin{eqnarray}
\label{def:X}
X_t & := & Z_t+\delta_{I_t},
\end{eqnarray}
clearly defines a ZRP $X=(X_t)_{t\ge 0}$ with $m$ particles starting from $X_0=z+\delta_i$. Given another site $j\in[n]$, we can of course enrich the above construction by adding a second walk $J=(J_t)_{t\ge 0}$ whose conditional evolution given $Z$ is dictated by the same time-varying rates $\left[r_u(Z_t+1)-r_u(Z_t)\right]\times G_{uv}$, but which now starts from $J_0=j$.  The formula
\begin{eqnarray}
\label{def:Y}
Y_t & := & Z_t+\delta_{J_t},
\end{eqnarray}
 defines a new ZRP $Y=(Y_t)_{t\ge 0}$ with $m$ particles, now starting from $Y_0=z+\delta_j$. From (\ref{def:X}-\ref{def:Y}), it is clear that the pair $(X,Y)$ resulting from this construction satisfies
\begin{eqnarray}
\label{ZRP:coupling}
\dist\left(X_t,Y_t\right) & = & \left\{ \begin{array}{ll}
1 & \textrm{if }I_t\ne J_t;\\
0 & \textrm{if }I_t= J_t.
\end{array}
\right.
\end{eqnarray}
Setting $x:=z+\delta_i$ and $y:=z+\delta_j$, we have thus constructed a coupling $(X_t,Y_t)$ of $P_t(x,\cdot)$ and $P_t(y,\cdot)$ such that $\dist(X_t,Y_t)\le\dist(x,y)$. To conclude, observe that any pair $(x,y)\in\cX^2$ with $\dist(x,y)=1$ can be written as $(x,y)=(z+\delta_i,z+\delta_j)$ for some $i\ne j\in[n]$ and some $z\in\dZ_+^n$ with $\sum_{i=1}^nz_i=m-1$. Moreover, our metric $\dist$ is generated by the set of such pairs, in the sense of (\ref{geodesic}). Thus, $P_t$ has non-negative sectional curvature, and  replacing $G$ with $G^\star$ yields the same conclusion for $P_t^\star$.\end{proof}
The existence of a  monotone coupling  between ZRPs with different numbers of particles  is of course a well known consequence of the rate monotonicity (\ref{assume:monotone}), but its interpretation in terms of sectional curvature seems to be new. The very same coupling actually also provides an estimate on the curvature of the ZRP, which is exactly what we need in order to apply Theorem \ref{th:main}. More precisely, the property (\ref{ZRP:coupling}) guarantees that for all $t\ge 0$, 
\begin{eqnarray}
\label{ZRP:curvature}
1-\cc(P_t) & \le & \max_{z,i,j}\PP_{z,i,j}(T>t),
\end{eqnarray}
where the maximum ranges over all possible choices for the initial triple $(z,i,j)$  appearing in the above proof, and where $T:=\inf\left\{t\ge 0\colon I_t=J_t\right\}$ denotes the coalescence time of our two time-inhomogeneous random walks $I$ and $J$. Note that we have specified the conditional distributions of $I$ and $J$ given $Z$, but not the way in which those two conditional distributions were actually coupled: the formula (\ref{ZRP:curvature}) is valid for any such coupling. To appreciate its strength, let us consider the important \emph{mean-field} case where the matrix $G$ has rank one:
\begin{eqnarray}
\label{def:MFZRP}
\forall i,j\in[n],\qquad G_{ij} & = & \nu_j.
\end{eqnarray} 
\begin{corollary}\label{co:ZRP}The mean-field  ZRP process satisfies the MLSI (\ref{MLSI}) with  $\kappa=\delta$, where 
\begin{eqnarray*}
\label{def:delta}
\delta & := & \min_{i\in[n],k\in[m]}\{r_{i}(k+1)-r_i(k)\}.
\end{eqnarray*}
\end{corollary}
\begin{proof}The mean-field condition (\ref{def:MFZRP}) ensures that, conditionally on $Z$, the random walks $I$ and $J$ jump to any given state $j\in[n]$ at a rate at least $\delta\nu_j$, regardless of their current position. Thus, they can be coupled so that their coalescence time $T$ is stochastically dominated by an Exponential  variable with rate $\delta$, and applying Theorem \ref{th:main} to $P_t$ concludes the proof.
\end{proof}
This result provides a final answer to a natural  question that has been the subject of several works \cite{MR2548501,MR2322692,MR3513606,MR4332696,conforti2022probabilistic}.
We emphasize that all prior lower bounds on the MLSI constant of the mean-field ZRP involved an additional dependency on the maximum rate increment 
\begin{eqnarray*}
\Delta & := & \max_{i\in[n],k\in[m]}\{r_{i}(k+1)-r_i(k)\},
\end{eqnarray*} and were vanishing in the $\Delta\to\infty$ limit. Let us also note that the applicability of (\ref{ZRP:curvature}) is by no means restricted to the mean-field case. As a concrete example, consider the case where $G$ is the transition matrix of simple random walk on the $n-$segment. Then, letting the two random walks $I,J$ evolve independently (conditionally on $Z$) until coalescence  ensures that their order is preserved, so that they must have coalesced by the time at which the lowest one hits $n$. Since this takes time $O_\PP(n^2/\delta)$, we deduce that the entropy decay occurs on the time-scale $O(n^2/\delta)$, which is sharp. Finally, we mention that a similar argument as the one used in Corollary \ref{co:ZRP} applies to  the heterogeneous Bernoulli-Laplace model, allowing us to  get rid of the dependency on the maximum rate in \cite[Theorem 5.1]{MR2548501}.
%

\bibliographystyle{plain}
\bibliography{draft}
\end{document}